\documentclass[12pt]{amsart}
\usepackage{pdfsync}
\usepackage{amssymb}
\usepackage{amsmath}
\usepackage{amsthm}
\usepackage{enumitem}
\usepackage{comment}
\usepackage[left=3cm,top=3.8cm,right=3cm]{geometry} % for an optimized layout
\usepackage{cite}
\usepackage[utf8]{inputenc}
\usepackage{verbatim}                                  % Enables the use of \begin{comment} ... \end{comment}

\usepackage{xcolor}
\usepackage[pagebackref]{hyperref}
\hypersetup{
   colorlinks,
    linkcolor={red!60!black},
    citecolor={blue!60!black},
    urlcolor={blue!90!black}
}

\newcommand{\e}{\varepsilon}
\renewcommand{\phi}{\varphi}
\DeclareMathOperator{\hdim}{dim_H}
\DeclareMathOperator{\pdim}{dim_P}
\DeclareMathOperator{\hlogdim}{dim^{\log}_H}
\DeclareMathOperator{\plogdim}{dim^{\log}_P}

\DeclareMathOperator{\supp}{supp}
\DeclareMathOperator{\diam}{diam}

\newtheorem{theorem}{Theorem}
\newtheorem{definition}[theorem]{Definition}
\newtheorem{proposition}[theorem]{Proposition}
\newtheorem{corollary}[theorem]{Corollary}

\newtheorem{lemma}[theorem]{Lemma}

\def\R{{\mathbb{R}}}
\def\Z{{\mathbb{Z}}}
\def\N{{\mathbb{N}}}

\def\C{{\mathbb{C}}}

\newcommand{\cD}{\mathcal{D}}

\newcommand{\RRP}{\mathsf{RRP}}

\makeatletter
\def\moverlay{\mathpalette\mov@rlay}
\def\mov@rlay#1#2{\leavevmode\vtop{%
   \baselineskip\z@skip\lineskiplimit-\maxdimen
     \ialign{\hfil$\m@th#1##$\hfil\cr#2\crcr}}}
\newcommand{\charfusion}[3][\mathord]{
    #1{\ifx#1\mathop\vphantom{#2}\fi
        \mathpalette\mov@rlay{#2\cr#3}
      }
    \ifx#1\mathop\expandafter\displaylimits\fi}
\makeatother

\numberwithin{theorem}{section}
\numberwithin{equation}{section}

\begin{document}
\thispagestyle{empty}

\title{Full measure universality for Cantor Sets}

\author{Pablo Shmerkin}
\address{Department of Mathematics, The University of British Columbia, 1984 Mathematics Road, Vancouver BC V6T 1Z2, Canada}
\email{pshmerkin@math.ubc.ca}

\author{Alexia Yavicoli}
\address{Department of Mathematics, The University of British Columbia, 1984 Mathematics Road, Vancouver BC V6T 1Z2, Canada}
\email{yavicoli@math.ubc.ca}


\thanks{The authors were each partially supported by NSERC Discovery Grants.}

\begin{abstract}
    We investigate variants of the Erd\H{o}s similarity problem for Cantor sets. We prove that under a mild Hausdorff or packing logarithmic dimension assumption, Cantor sets are not full measure universal, significantly improving the known fact that sets of positive Hausdorff dimension are not measure universal. We prove a weaker result for all Cantor sets $A$: there is a dense $G_\delta$ set of full measure $X\subset\mathbb{R}^d$, such that for any bi-Lipschitz function $f:\mathbb{R}^d\to \mathbb{R}^d$, the set of translations $t$ such that $f(A)+t\subseteq X$ is of measure zero. Equivalently, there is a null set $B\subset\mathbb{R}^d$ such that $\mathbb{R}^d\setminus (f(A)+B)$ is null for all bi-Lipschitz functions $f$.

\end{abstract}

\maketitle

%\tableofcontents

\section{Introduction and main results}

A set $A\subset \R$ is termed \emph{measure universal} if it can be affinely embedded into all sets of positive Lebesgue measure. More precisely, $A$ is measure universal if for every $E\subset\R$ of positive Lebesgue measure, there exists a non-constant affine map $f$ such that $f(A)\subset E$.

Erd\H{o}s famously posed the problem of showing that no infinite set is measure universal (a simple density argument shows that all finite sets are measure universal). Despite a significant amount of work on the problem (see eg.\cite{Falconer84, Eigen85, Bourgain87, Kolountzakis97} as well as the surveys \cite{Svetic00, JLM25}), the conjecture remains wide open. For example, it is not known whether $\{ 2^{-n}:n\in\N\}$ is measure universal.

Quite strikingly, it is not even known whether all uncountable compact sets fail to be measure universal. The goal of the present paper is to make progress on this question, both in the setting of the Erd\H{o}s similarity problem and some natural variants of it.

We review previous work on the subject. Note that an uncountable compact set (and even an uncountable Borel set) contains a topological Cantor set. Thus, it is equivalent to consider the problem for Cantor sets. Let $\mathcal{F}$ be a translation-invariant family of functions from $\R^d\to\R^d$. We say that $A\subset\R^d$ is \emph{$\mathcal{F}$-full measure universal} if for all sets $E\subset\R^d$ of full measure, there exists $f\in\mathcal{F}$ such that $f(A)\subset E$. If $\mathcal{F}$ is the family of non-constant affine maps, we drop it from the notation. Obviously, if a set is not full measure universal, it is not measure universal, and countable sets are trivially full measure universal.

J.~Gallagher, C-K.~Lai and E.~Weber \cite[Theorem 1.4]{GLW23} proved that no Cantor sets in $\R^d$ are universal in a topological sense (which by itself does not say anything about measure universality). They also showed that Cantor sets of positive Newhouse thickness in the line are not full measure universal \cite[Theorem 1.5]{GLW23}. It was noted by D-J.~Feng, C-K.~Lai and Y.~Xiong in \cite{FLX24} that the argument extends to show that sets of positive thickness in the line are not full measure universal for the class of bi-Lipschitz functions. Sets of positive thickness also have positive Hausdorff dimension. It was observed in \cite[Corollary 3.8]{JLM25} that a result of P.~Shmerkin and V.~Suomala \cite[Theorem 13.1]{SS18} implies that sets of positive Hausdorff dimension in $\R$ are not full measure universal. In fact, the same argument shows that sets of positive Hausdorff dimension in $\R^d$ are not full-measure universal. Regarding sets of zero dimension, the only result we are aware of is that of M.~Kolountzakis \cite{Kolountzakis23}, who proved that certain symmetric Cantor sets are not measure universal. This is a rather restricted class.

By the following well-known argument, full measure universality is closely related to nonempty interior in sumsets.
\begin{lemma} \label{lem:universality-sumset}
Let $\mathcal{F}$ be a translation-invariant family of non-constant functions from $\R^d\to\R^d$. Then, a set $A\subset\R^d$ is not $\mathcal{F}$-measure universal if and only if there exists a Lebesgue-null set $B\subset\R^d$ such that $f(A)+B=\R^d$ for all $f\in\mathcal{F}$.
\end{lemma}
\begin{proof}
        If $A$ is not $\mathcal{F}$-measure universal, then there exists a set $B$ of zero measure such that $f(A)\not\subset \R^d\setminus B$ for all $f\in\mathcal{F}$. Fix $t\in\R^d$ and let $\tilde{f}(x)=f(x)+t$ where $f$ fixes $0$. We know that $\tilde{f}(A)\cap B\neq \emptyset$, which is the same as saying that $t\in f(A)-B$. Hence, $f(A)-B=\R^d$. The argument can be easily reversed to give the converse implication.
\end{proof}

Constructions of M.~Talagrand \cite{Talagrand76} and of P.~Erd\H{o}s, K.~Kunen and R.D.~Mauldin \cite{EKM81} show that for any uncountable compact set $A$ in the line, there exists a compact set $B$ of zero measure such that $A+B$ has nonempty interior. If one could take the same set $B$ to work for all non-constant linear functions $f$, the above lemma would imply that $A$ is not full measure universal. The constructions in \cite{Talagrand76, EKM81} do not have this property. However, we adapt some of their ideas.

In our first main result, we show that sets of large logarithmic Hausdorff or Packing dimension are not full measure universal. This class includes all sets of positive packing dimension (which is a weaker condition than positive Hausdorff dimension), as well as many sets of zero Hausdorff or packing dimension. We recall the definition of generalized Hausdorff and Packing measures in section \ref{sectionGHPM}.
\begin{definition}
    The \emph{logarithmic Hausdorff (resp. packing) dimension}  of a set $A\subset\R^d$, denoted  $\hlogdim(A)$ (resp. $\plogdim(A)$) is defined as the infimum of all $s>0$ such that $\mathcal{H}^{\log^{-s}(1/x)}(A)=0$ (resp. $\mathcal{P}^{\log^{-s}(1/x)}(A)=0$). In both cases, if the set in question is empty, we define the logarithmic dimension to be $\infty$.
\end{definition}
Note that a set of positive Hausdorff (resp. packing) dimension has infinite logarithmic dimension (and so do many sets of zero Hausdorff or packing dimension).
\begin{theorem} \label{thm:full-sumset-Erdos}
    For any Borel set $A\subset\R^d$ such that either $\plogdim(A)>2$ or $\hlogdim(A)>1$,  there exists an $F_\sigma$ set $B$ of zero Lebesgue measure such that $f(A)+B=\R^d$ for all non-constant affine functions $f:\R^d\to\R^d$.

    In particular, by Lemma \ref{lem:universality-sumset}, $A$ cannot be affinely embedded into the full measure and dense $G_\delta$ set $\R^d\setminus (-B)$.
\end{theorem}

In the real line, we also show that the above also holds for all non-constant polynomial functions, see Corollary \ref{cor:polynomial-images}.

If $A\subset\R^d$ is an arbitrary Cantor set, we are unable to show that $f(A)+B=\R^d$ for a null set $B$ and all non-constant affine functions $f$, but we can prove that $f(A)+B$ has full Lebesgue measure, even for all bi-Lipschitz functions $f$. In fact, we prove something more general.
\begin{theorem} \label{thm:full-measure-sumset}
    For any dimension function $\phi$, there exists an $F_\sigma$-set $B$ of zero measure such that the following holds:
    \[
        \bigl|\R^d\setminus (A+B)\bigr| = 0
    \]
    for all Borel sets $A$ with $\mathcal{H}^\phi(A)>0$.
\end{theorem}

It is well-known that for all Cantor sets $A$, there is a dimension function $\phi$ such that $\mathcal{H}^\phi(A)>0$, see e.g. \cite[Theorem 35]{Rogers70}. The class of sets $A$ such that $\mathcal{H}^{\phi}(A)>0$ is invariant under bi-Lipschitz maps. Thus, we get the following corollary.
\begin{corollary} \label{cor:full-measure-sumset}
    Let $A\subset\R^d$ be a Cantor set. There exists an $F_\sigma$-set $B$ of zero measure such that
    \[
        \bigl|\R^d\setminus (f(A)+B)\bigr| = 0
    \]
    for all bi-Lipschitz functions $f:\R^d\to \R^d$.

    In particular, there is a dense $G_\delta$ set of full measure $X\subset\R^d$ such that for any bi-Lipschitz function $f:\R^d\to \R^d$, the set of translations $t$ such that $f(A)+t\subseteq X$ is of measure zero.
\end{corollary}
The latter claim follows from the same argument as in Lemma \ref{lem:universality-sumset}. Kolountzakis \cite{Kolountzakis97} proved that for any infinite set $A\subset\R$ there exists a set $X$ of positive measure such that for all $r\neq 0$, the set of translations $t$ such that $rA+t\subseteq X$ is of measure zero. Corollary \ref{cor:full-measure-sumset} recovers and strengthens this result when $A$ is a Cantor set.

Theorem \ref{thm:full-measure-sumset} in the case of sets of positive Hausdorff dimension follows from the existence of sets of zero measure and arbitrarily large Hausdorff dimension that support measures with fast Fourier decay. This is likely known to experts, but we recall the argument in Lemma \ref{lem:positive-measure-positive-dim} since we were not able to find it in the literature. In order to establish Theorem \ref{thm:full-measure-sumset} for arbitrary dimension functions, we also rely on Fourier decay, but applied locally rather than globally. The construction relies on classical number-theoretic results closely related to Gauss sums.

The paper is organized as follows. In Section \ref{sec:largeness}, we review general Hausdorff and packing measures and introduce a related notion of largeness that turns out to be more natural in the context of full measure universality. In Section \ref{sec:nonempty-interior}, we prove a general result on nonempty interior of sumsets, Theorem \ref{thm:technical-nonemptyinterior}, and deduce Theorem \ref{thm:full-sumset-Erdos} and other applications. In Section \ref{sec:full-measure}, we prove Theorem \ref{thm:full-measure-sumset}.

\subsection*{Acknowledgements}
The authors would like to express their gratitude to the Banff International Research Station for Mathematical Innovation and Discovery (BIRS) for providing an inspiring environment where this project was initiated. Alexia Yavicoli would like to thank Chun-Kit Lai, Károly Simon and Ville Suomala for helpful discussions at BIRS. Both authors thank De-Jun Feng for valuable discussion on the Erd\H{o}s similarity problem while visiting CUHK. We thank an anonymous referee for a careful reading of the manuscript and many helpful comments.

\section{Hausdorff and packing measures and a related notion of largeness}

\label{sec:largeness}

\subsection{Notation}

We let $\mathcal{D}_m$ denote the family of closed dyadic cubes in $\R^d$ of side length $2^{-m}$ (the ambient dimension $d$ is implicit from context). For a set $A\subset\R^d$, we let $\mathcal{D}_m(A)$ denote the family of cubes in $\mathcal{D}_m$ that intersect $A$. Let $\sigma(\mathcal{D}_m)$ denote the $\sigma$-algebra generated by the cubes in $\mathcal{D}_m$, that is, the collection of all sets which are unions of cubes in $\mathcal{D}_m$.

Given a cube $Q$ and a constant $C>0$, the notation $CQ$ refers to the cube with the same centre as $Q$ and side length $C\cdot \text{side length}(Q)$.

The $\delta$-neighbourhood of a set $A\subset \R^d$ is denoted by $A^{(\delta)}=\bigcup_{x\in A} B(x,\delta)$.

We write $[m]:=\{0,\ldots,m-1\}$.

The canonical vectors of $\R^d$ are denoted by $e_1,\ldots,e_d$.

\subsection{Hausdorff metric}
\label{subsec:Hausdorff-metric}

Given two compact sets $K_1,K_2\subset \R^d$, the Hausdorff distance between them is defined as
\[
    \mathsf{H}(K_1,K_2):=\inf\{\delta>0: K_1\subset K_2^{(\delta)} \text{ and } K_2\subset K_1^{(\delta)} \}.
\]
The space of nonempty compact subsets of $\R^d$, equipped with the Hausdorff metric, is a complete metric space.

\subsection{General Hausdorff and packing Measures}\label{sectionGHPM}
We recall the definition of Hausdorff and packing measure associated with an arbitrary dimension function.
\begin{definition}
    The space of \emph{dimension} or \emph{gauge} functions is defined as
    \begin{equation*}
        \Phi:= \left\{
        \begin{array}{ll}
            \phi:\R_{\geq 0}\to\R_{\geq 0} & \mathrm{ :\ } \phi(t)>0 \text{ if } t>0,
            \phi(0)=0,                                                                            \\
                                        & \mathrm{   } \text{ increasing and right-continuous}
        \end{array}
        \right\}.
    \end{equation*}
\end{definition}
This set is partially ordered, with the order defined by
\[
    \phi_2 \prec \phi_1 \text{ if } \lim_{x \to 0^+} \frac{\phi_1(x)}{\phi_2(x)} =0.
\]

\begin{definition}
    The outer \emph{Hausdorff measure} associated with $\phi\in \Phi$ is
    \[
        \mathcal{H}^\phi(E):=\lim_{\delta\to 0} \inf \left\{ \sum_{i}\phi(\diam(U_i)) : \diam(U_i)<\delta, \, E\subset\bigcup_i U_i  \right\}.
    \]
\end{definition}
For any $\phi \in \Phi$,  $\mathcal{H}^\phi$ is an outer measure and a Borel-regular measure. This definition generalizes the outer $\alpha$-dimensional Hausdorff measure, which corresponds to the particular case $\phi(x):=x^{\alpha}$. The order relation implies that $x^s \prec x^t$ if and only if $s<t$.

A \emph{packing} of a set $E$ is a collection of disjoint balls with centres in $E$. A $\delta$-packing of $E$ is a packing in which all radii are less than $\delta$.
\begin{definition}
    Let $\phi\in \Phi$. The \emph{packing pre-measure} associated with $\phi$ is defined as
    \[
        \mathcal{P}_0^\phi(E):=\lim_{\delta\to 0} \sup \left\{\sum_i \phi(r_i): B(x_i,r_i) \text{ is a $\delta$-packing of $E$} \right\}.
    \]
    The outer \emph{packing measure} associated with $\phi$ is defined as
    \[
        \mathcal{P}^\phi(E):=\inf \left\{ \sum_i \mathcal{P}_0^\phi(E_i): E\subset \cup_i E_i \right\}.
    \]
\end{definition}
The packing pre-measure is monotone, but it is not an outer measure, hence why an extra step is needed in the definition of packing measure. The packing outer measure is an outer measure and a Borel-regular measure. In Euclidean space, the following inequality between Hausdorff and packing measure holds:
\begin{lemma}
    Suppose $\phi\in \Phi$ satisfies $\phi(t) \le C \phi(t/2)$ for some constant $C>0$ and all $t>0$. Then, $\mathcal{H}^\phi(A)\le \mathcal{P}^\phi(A)$ for all sets $A\subset\R^d$.
\end{lemma}
See \cite{Cutler95} for proofs of these facts and further background on packing measures.

\subsection{A notion of largeness}

We introduce a notion of largeness of a set which, as we will see shortly, is closely related to both Hausdorff and packing measures, but different from either. Let $|A|_{\delta}$ denote the $\delta$-covering number of $A$, that is, the smallest number of $\delta$-balls needed to cover $A$. Recall that the $\delta$-packing number of $A$, denoted by $|A|_{\delta}^{\text{pack}}$, defined as the largest number of disjoint $\delta$-balls with centres in $A$, satisfies
\begin{equation} \label{eq:packing-covering}
      C_d^{-1} |A|_{\delta} \le |A|_{2\delta} \le |A|_{\delta}^{\text{pack}} .
\end{equation}
Indeed, if $\{ B(x_i,\delta) \}$ is a maximal packing of $A$, then $\{ B(x_i,2\delta)\}$ is a covering of $A$ (otherwise, we could add $B(x,\delta)$ to the packing, for $x\in A\setminus \cup_i B(x_i,\delta)$). On the other hand, a ball of radius $2\delta$ can be covered by $C_d$ balls of radius $\delta$, where $C_d$ is a constant depending only on the dimension $d$ (for example, $C_d=5^d$ can be checked to work).

\begin{definition} \label{def:large-set}
    Let $A\subset\R^d$ be a Borel set and let $N(\delta):(0,1]\to\N$ be a function such that
    \[
        \lim_{\delta\to 0^{+}} N(\delta) = \infty.
    \]

    We say that $A$ is \emph{$N$-large} if there exists a compact set $A'\subset A$ such that for every cube $Q$ such that $A'\cap Q^{\circ}\neq\emptyset$,
    \[
       \limsup_{\delta\to 0^{+}} \frac{|A'\cap Q^{\circ}|_{\delta}}{N(\delta)} \ge 1.
    \]

    If $\delta_k, N_k$ are sequences in $(0,1]$ and $\N$ converging to $0$ and $\infty$ respectively, we say that $A$ is \emph{$(N_k,\delta_k)$-uniformly large} if it contains a compact set $A'$ with the following property:
    \[
        |A'\cap Q|_{\delta_k} \ge N_k
    \]
    for all cubes $Q\in\mathcal{D}_k(A')$.
\end{definition}

Both the statement and the proof of the following lemma are closely related to (the proof of) Frostman's Lemma.
\begin{lemma} \label{lem:large-Hausdorff-packing}
    Let $A\subset\R^d$ be a Borel set, and let $\phi\in\Phi$.
    \begin{enumerate}[label=\textup{(\roman*)}]
        \item  \label{it:assumption-H} If $\mathcal{H}^\phi(A)>0$, then $A$ is $N$-large for any function $N$ such that
        \[
            \limsup_{\delta\to 0^+} \phi(\delta) \cdot N(\delta)=\infty.
        \]
        \item \label{it:assumption-P} If $\mathcal{P}^\phi(A)>0$, then $A$ is $N$-large for
        \[
            N(\delta) = \frac{1}{\log (2\delta^{-1})\log^2\log(2\delta^{-1})\cdot \phi(4\delta)},
        \]
        provided that $\lim_{\delta\to 0^+}N(\delta)=\infty$.
    \end{enumerate}
\end{lemma}
\begin{proof}
    Without loss of generality, $A$ is compact and contained in $[0,1]^d$. Let $\mathcal{X}\in\{\mathcal{H},\mathcal{P}\}$. Let $k$ be the smallest dimension for which there is a $k$-affine plane $V$ such that $\mathcal{X}^\phi(A\cap V)>0$. Replacing $\R^d$ by $V$ and $A$ by $A\cap V$, we may assume that $\mathcal{X}^\phi(V)=0$ for all proper affine planes $V$.

    Let $\mathcal{D}^{\mathcal{X}}(A)$ denote the family of dyadic cubes $Q$ such that $\mathcal{X}^\phi(A\cap Q)>0$. Note that, by our assumption, $\mathcal{X}^\phi(A\cap Q^{\circ})=\mathcal{X}^\phi(A\cap \overline{Q})$. Let $A'_m$ be the union of the closure of the cubes in $\mathcal{D}^{\mathcal{X}}(A)$ of side length $2^{-m}$. Note that $A'_m$ is decreasing in $m$, and set
    \[
       A' = \bigcap_{m=1}^\infty A'_m.
    \]

    Let $Q$ be a cube such that $A'\cap Q^{\circ}\neq \emptyset$. By construction, we have $\mathcal{X}^\phi(A'\cap Q^{\circ})>0$.
    
    %$\mathcal{X}^{\phi}(A'\cap Q^{\circ})>0$. Then, there are $m$ and a cube $Q'\in \mathcal{D}^{\mathcal{X}}(A)$ of side length $2^{-m}$ such that $Q'\subset Q$. By the definition of $A'_m$, we have $\mathcal{X}^\phi(A'\cap Q')>0$.

    We consider the case $\mathcal{X}=\mathcal{H}$ first. Let $\{ B(x_i,\delta)\}_{i\in I}$ be a cover of $A'\cap Q$ by balls of radius $\delta$. Then, by definition of Hausdorff measures, if $\delta$ is small enough we have
    \[
        |I|\cdot \phi(\delta) \ge \frac{1}{2}\mathcal{H}^\phi(A'\cap Q^{\circ}) =: c_Q > 0.
    \]
    On the other hand, by the assumption \ref{it:assumption-H}, there is a sequence $\delta_j\to 0$ such that $\phi(\delta_j) N(\delta_j) >c_Q$ for all $j$. Thus, for $j$ large enough we have
    \[
     |A'\cap Q^{\circ}|_{\delta_j} \ge \frac{c_Q}{\phi(\delta_j)}  \ge N_j.
    \]
    This yields the first claim.

    We turn to packing measures. Suppose $A' \cap Q^{\circ} \neq \emptyset$. Then, by construction of $A'$ and the definition of packing pre-measure,
    \[
        \mathcal{P}_0^\phi(A'\cap Q^{\circ}) \ge \mathcal{P}^\phi(A'\cap Q^{\circ})>0.
    \]
    Let $k_0$ be chosen so that
    \[
        \sum_{k=k_0}^\infty \frac{1}{k\log^2 k} < \frac{1}{2}\mathcal{P}_0^\phi(A'\cap Q^{\circ}).
    \]
    By the definition of packing pre-measure, if $k_0$ is large enough then there exists a $2^{-k_0}$-packing $\{ B(x_i,r_i) \} $ of $A'\cap Q^{\circ}$ such that
    \[
        \sum_{i} \phi(r_i) \ge \frac{1}{2}\mathcal{P}_0^\phi(A'\cap Q^{\circ})>0.
    \]
    Let $I_k = \{ i: r_i \in [2^{-k}, 2^{-k+1}) \}$. By the pigeonhole principle, there exists $k\ge k_0$ such that
    \[
        |I_k| \cdot \phi(2^{1-k}) \ge \sum_{i\in I_k} \phi(r_i) \ge \frac{1}{k\log^2 k}.
    \]
    Since the points $x_i, i \in I_k$ are $2^{-k}$-separated, a cover of $\{x_i\}_{i\in I_k}$ by balls of radius $2^{-k-1}$ has cardinality $\ge |I_k|$.  We conclude that
    \[
        |A'\cap Q^{\circ}|_{2^{-k-1}} \ge |I_k|\ge \frac{1}{k\log^2 k\cdot  \phi(2^{1-k})},
    \]
    yielding the claim.
\end{proof}

The following refinement of the first part of the above lemma will be useful later.
\begin{lemma} \label{lem:Hausdorff-unif-large}
    Let $\phi\in\Phi$, let $\eta>0$, and let $\delta_k$ be a sequence such that
    \[
        \lim_{k\to\infty} 2^{3kd+1} \phi(\delta_k) = 0.
    \]
    Then, any set $A\subset [0,1]^d$ with $\mathcal{H}^\phi(A)\ge \eta$ is $(N_k,\delta_k)$-uniformly large for
    \[
      N_k= \frac{\eta}{2^{3kd+1} \cdot \phi(\delta_k)}.
    \]
\end{lemma}
\begin{proof}
    Let $\eta_k=2^{-3kd}\eta$, and note that
    \[
        \sum_{\ell=k+1}^\infty 2^{(\ell-k)d}\eta_{\ell} = 2^{-3kd}\eta \sum_{\ell=1}^\infty 2^{-2d\ell} \le \frac{\eta_{k}}{2}, \quad k\ge 0.
    \]

    We construct a nested sequence of compact sets $A_k$ as follows.
    Let $A_0=A$. Once $A_k$ has been constructed, let $A_{k+1}$ be the union of the (closed) cubes in $\mathcal{D}_{k+1}(A_k)$ such that $\mathcal{H}^\phi(A_k\cap Q)> \eta_{k+1}$. Let $A'=\bigcap_{k=1}^\infty A_k$.

    Let $Q\in\mathcal{D}_k(A')$, $k\in\N_0$. Then, $\mathcal{H}^\phi(A_k\cap Q)\ge \eta_k$. Since $A_k$ intersects $\le 2^{(\ell-k)d}$ cubes of side length $2^{-\ell}$ for $\ell>k$, and each of the cubes removed from $A_{\ell}$ has $\mathcal{H}^\phi$-measure at most $\eta_{\ell}$, we have
    \[
        \mathcal{H}^\phi(A'\cap Q) \ge \eta_k - \sum_{\ell=k+1}^\infty 2^{(\ell-k)d}\eta_{\ell} \ge \frac{\eta_k}{2}.
    \]
    Now if $\{ B(x_i,\delta_k)\}_{i\in I}$ is a cover of $A'\cap Q$ by balls of radius $\delta_k$, where $Q\in\mathcal{D}_k$ and $A'\cap Q\neq \emptyset$, we have
    \[
        2^{-3kd-1}\eta \le \mathcal{H}^\phi(A'\cap Q) \le |I|\cdot \phi(\delta_k).
    \]
    Thus, $|I|\ge 2^{-3kd-1}\eta/\phi(\delta_k)$, and the claim follows.
\end{proof}

\section{Non-empty interior of sumsets and full measure universality}

\label{sec:nonempty-interior}

\subsection{A general inductive approach}

In this section we derive a general approach to proving that sumsets have nonempty interior, which is inspired by \cite{Talagrand76, EKM81}.
\begin{definition}
    A \emph{rectangle} in $\R^d$ is a set of the form $Q=\prod_{i=1}^d [a_i,b_i]$, where $a_i,b_i\in\R$ and $a_i<b_i$ for all $i$.

    If $R$ is a rectangle, $CR$ denotes the rectangle with the same centre and side lengths equal to $C$ times the corresponding side lengths of $R$.

    A family of rectangles is \emph{non-overlapping} if the interiors of the rectangles are pairwise disjoint.

    An \emph{elementary set} in $\R^d$ is a finite union of rectangles.
\end{definition}

\begin{definition}
    Let $A$ be a compact subset of $\R^d$ and let $\mathcal{F}$ be a family of functions from $A$ to $\R^d$.
    We say that $(A,\mathcal{F})$ has the \emph{recursive rectangles property} or $\RRP$ if there exists a rectangle $R$ (called the initial rectangle) such that following holds.

    For every rectangle $Q\subset R$, for every $a\in A$, and for every $\e>0$, there exist a finite subset $A'\subset A$ and an elementary set $T\subset 3Q$ such that
    \begin{enumerate}[label=\textup{(\roman*)}]
        \item $f(a)+Q\subset f(A')+T$ for all $f\in\mathcal{F}$.
        \item $|T|\le \e |Q|$.
    \end{enumerate}
    %We say that $\mathcal{A}$ has the \emph{weak recursive rectangles property} or $\WRRP$ if for the above holds with (i) replaced by
    %\begin{enumerate}[label=\textup{(\roman*')}]
    %    \item $\displaystyle{\big|f(a)+Q\setminus \left(f(A')+T\right)\big|\le \e|Q|}$ for all $f\in\mathcal{F}$.
    %\end{enumerate}
\end{definition}

\begin{lemma} \label{lem:nonemptyinterior}
    Suppose $(A,\mathcal{F})$ has the $\RRP$ property with initial rectangle $R$. Then for every $a_0\in A$ there is a compact set $K$ with $|K|=0$ such that
    \begin{equation} \label{eq:nonemptyinterior}
        f(a_0)+R\subset f(A)+K \quad\text{for all } f\in\mathcal{F}.
    \end{equation}
\end{lemma}
\begin{proof}

    Let $a_0\in A$ be arbitrary. Recall that $K^{(\delta)}$ denotes the $\delta$-neighbourhood of $K$. We will inductively construct a sequence of elementary sets $K_j\subset R$, numbers $0<\delta_j\le 2^{-j}$, and finite sets $A'_j\subset A$ such that, for all $j\ge 0$,
    \begin{enumerate}[label=(\alph*)]
        \item \label{it:a-NEI} $K_{j+1}\subset K_j^{(\delta_j)}$,
        \item \label{it:b-NEI} $|K_{j}|\le 5^{-d} 2^{-j}$, and $\displaystyle{\left|K_{j}^{(2\delta_j)}\right|\le  2^{-j}}$,
        \item \label{it:c-NEI} $f(a_0) +R\subset f(A'_j)+K_j$ for all $f\in\mathcal{F}$.
    \end{enumerate}
    If we let $\Delta_{j}=\sum_{\ell=j}^\infty\delta_{\ell}$, we see from \ref{it:a-NEI} that the sequence of compact sets $K_j^{(\Delta_j)}$ is decreasing. Let
    \[
        K = \bigcap_{j=1}^\infty K_j^{(\Delta_j)}.
    \]
    Since $\delta_j\le 2^{-j}$, we have $\Delta_j\le 2\delta_j$, and it follows from \ref{it:b-NEI} that $|K|=0$. Finally, we get from \ref{it:c-NEI} that $f(a_0)+R \subset f(A)+K_j$ for all $j$ and $f\in\mathcal{F}$, and therefore \eqref{eq:nonemptyinterior} holds.

    Set $K_0=R$, and $A'_0=\{a_0\}$ to begin the construction. The parameter $\delta_j$ will be defined in the step $j+1$; we will verify the second part of \ref{it:b-NEI} for $j$ in the step $j+1$ as well.

    Suppose $K_j$, $\delta_j$, $A'_j$ have been constructed. We can write $K_j= \cup_{s\in S} Q_s$, where $S$ is a finite set and the $Q_s$ are non-overlapping rectangles. By further splitting each $Q_s$ into many sufficiently small nearly congruent pieces, we may assume that, for some number $\delta_{j+1}\le 2^{-(j+1)}$, all elements of $Q_s$ are contained in a cube of diameter $\delta_{j}$ and contain a cube of diameter $\delta_{j}/2$.

    By the $\RRP$, for each $s\in S$ and $a\in A'_j$, there are an elementary set $T_{s,a}$ and a finite set $A'_{s,a}\subset A$ such that:
    \begin{enumerate}[label=(\Alph*)]
        \item \label{it:i-NEI} $T_{s,a}\subset 3 Q_s$,
        \item \label{it:ii-NEI} $f(a)+Q_s\subset f(A'_{s,a})+T_{s,a}$ for all $f\in\mathcal{F}$.
        \item \label{it:iii-NEI} $\displaystyle{|T_{s,a}|\le \frac{|Q_s|}{2\cdot 5^d\cdot |A'_j|}}$.
    \end{enumerate}

    Let
    \begin{align*}
        A'_{j+1} &= \bigcup_{a\in A'_j, s\in S} A'_{s,a}\\
        K_{j+1} &=\bigcup_{a\in A'_j, s\in S} T_{s,a}.
    \end{align*}
    Since $T_{s,a}\subset 3 Q_s$, we have
    \[
    K_{j+1}\subset \bigcup_{s\in S} 3 Q_s \subset K_j^{(\delta_j)}.
    \]
    Thus, \ref{it:a-NEI} holds for $j+1$.

    By the inductive assumption and \ref{it:ii-NEI}, for all $f\in\mathcal{F}$ we have
    \[
       f(a_0) + R \subset \bigcup_{a\in A'_{j}, s\in S} f(a)+Q_s \subset  \bigcup_{a\in A'_{j}, s\in S} f(A'_{s,a})+T_{s,a} \subset f(A'_{j+1})+K_{j+1}.
    \]
    This shows that \ref{it:c-NEI} holds for $j+1$.

    We have
    \[
        |K_{j+1}| \le \sum_{s\in S}\sum_{a\in A'_j} |T_{s,a}| \le \frac{1}{2\cdot 5^d} \sum_{s\in S} |Q_s| = \frac{1}{2}|K_j|\le 5^{-d}\cdot 2^{-j-1}.
    \]
    Thus, the first part of \ref{it:b-NEI} holds for $j+1$. Finally, since each $Q_s$ has diameter at least $\delta_j/2$, we see that  $K_{j}^{(2\delta_j)}$ is contained in $\bigcup_{s\in S} 5 Q_s$. Therefore,
    \[
    \left|K_{j}^{(2\delta_j)}\right| \le 5^d\sum_{s\in S}|Q_s| = 5^d |K_{j}| \le 2^{-j},
    \]
    using the inductive assumption. This verifies the second part of \ref{it:b-NEI} for $j$.
\end{proof}

\subsection{A combinatorial lemma and its continuous counterpart}

  The ideas in this section are related to \cite[Section 3]{KeletiMathe22}. The following combinatorial lemma is an adaptation from an argument of I.~Ruzsa \cite[Lemma 5.2]{Ruzsa16}.
\begin{lemma} \label{lem:sum-combinatorial}
    Let $N\in\mathbb{N}$, let $\mathcal{A}$ be a collection of subsets of $\Z_N^d$ and fix $\e>0$.

    If
    \[
        |A|> \frac{2}{\e}\log\left(|\mathcal{A}|\cdot N^d\right) \quad\text{for all }A\in\mathcal{A},
    \]
    then there exists $B\subset \Z_N^d$ with $|B|\le \e N^d$ such that
    \[
        A+B = \Z_N^d \quad\text{for all } A\in\mathcal{A}.
    \]
\end{lemma}
\begin{proof}
    Write
    \[
        K = \frac{2}{\e}\log\left(|\mathcal{A}|\cdot N^d\right) .
    \]
    Let $B$ be random subset of $\mathbb{Z}_N^d$ of size $|B|=\lfloor \e N^d \rfloor\ge \e N^d/2$ (that is, $B$ is chosen uniformly among all such subsets). Then, for any fixed $A\in\mathcal{A}$ and $x\in \mathbb{Z}_N^d$, we have
    \begin{align*}
        \mathbf{P}(x \notin  A+B) &= \mathbf{P}\bigl(B\subset \mathbb{Z}_N^d\setminus (A-x)\bigr)  = {N^d-|A|\choose |B|}/{N^d\choose |B|} \\
        &\le \left(1-\frac{|A|}{N^d}\right)^{|B|} \le \left(1-\frac{|A|}{N^d}\right)^{\e N^d/2} < e^{-\e K/2},
    \end{align*}
    using the inequality $1-t\le e^{-t}$ applied with $t=|A|/N^d>K/N^d$. Thus,
    \[
        \mathbf{E}\left|\bigcup_{A\in\mathcal{A}}\mathbb{Z}_N^d\setminus (A+B) \right| \le \sum_{A\in\mathcal{A}} \mathbf{E}|\mathbb{Z}_N^d\setminus (A+B)| < |\mathcal{A}|\cdot N^d\cdot e^{-\e K/2} = 1,
    \]
    by the choice of $K$. Hence, there exists a realization of $B$ such that $|\Z_N^d\setminus (A+B)|=0$ for all $A\in\mathcal{A}$, as claimed.
\end{proof}

\begin{corollary}  \label{cor:sum-combinatorial}
    Let $N\in\mathbb{N}$, let $\mathcal{A}$ be a collection of subsets of $[N]^d$ and fix $\e>0$.

    If
    \[
        |A|> \frac{2}{\e}\log\left(|\mathcal{A}|\cdot N^d\right) \quad\text{for all }A\in\mathcal{A},
    \]
    then there exists $B\subset \{ -N,\ldots, N-1\}^d$ with $|B|\le 2^d \e N^d$, such that
    \[
        A+B \supseteq [N]^d \quad\text{for all } A\in\mathcal{A}.
    \]
\end{corollary}
\begin{proof}
    Seeing $A$ as a subset of $\Z_N^d$, we apply Lemma \ref{lem:sum-combinatorial} to get a set $\tilde{B}\subset \Z_N^d$ with $|\tilde{B}|\le \e N^d$ such that $A+\tilde{B}=\Z_N^d$ for all $A\in\mathcal{A}$. Then, we can take $B$ to be the union of the translates of $\tilde{B}$ by all vectors in $\Z^d$ with coordinates $-N$ or $0$.
\end{proof}

Recall the definition of the Hausdorff metric $d_H$ from \S\ref{subsec:Hausdorff-metric}.
\begin{lemma} \label{lem:discrete-to-continuous}
    Let $\e>0$ and let $\delta\in 2^{-\N}$ be a small enough dyadic number (depending on $\e$). Let $\mathcal{A}$ be a family of compact subsets of $[0,1]^d$.  Denote the $\delta$-covering number of $\mathcal{A}$ in the Hausdorff metric by $|\mathcal{A}|_{\delta}$. Suppose that
    \[
        |A|_{\delta}> \frac{18^d}{\e}\log\left(\delta^{-1}\cdot|\mathcal{A}|_{\delta}\right) \quad\text{for all }A\in\mathcal{A}.
    \]
    Then, there exists a set $B\subset [-1,1]^d$ which is a union of $\delta$-dyadic cubes, such that $|B|\le \e$ and
    \[
        A+B \supset [0,1]^d \quad\text{for all }A\in\mathcal{A}.
    \]
\end{lemma}
\begin{proof}
    Let $m=\delta^{-1}$. Given a vector $v$ with integer coordinates, let
    \begin{equation} \label{eq:Qv}
        Q_v = \frac{1}{m}v + \left[0,\frac{1}{m} \right]^d.
    \end{equation}
    Given a set $A\subset [0,1]^d$, let $\Delta[A]$ be the set obtained by replacing each element in $A\cap Q_{a'}$ by $a'$. Then, $\Delta[A]\subset [m]^d$. Note that each $Q_{a'}$ can be covered by $2^d+1$ balls of radius $\delta$ (one centred at the centre of the cube and one centred at each of the corners). Hence, we have
    \[
        |\Delta[A]| \ge (2^d+1)^{-1} |A|_{\delta}> 3^{-d}|A|_{\delta}.
    \]
    By Corollary \ref{cor:sum-combinatorial}, there exists $B'\subset \{ -m,\ldots, m-1\}^d$ with $|B'|\le \tfrac{1}{2}\e 3^{-d}\cdot m^d$, such that
    \[
        \Delta[A]+B' \supset [m]^d \quad\text{for all } A\in\mathcal{A}.
    \]
    Let
    \begin{equation*}
        \begin{split}
        \tilde{B} &=  \bigcup_{b\in B'} \bigcup_{b'\sim b} Q_{b'},\\
        B &= [-1,1]^d\cap \tilde{B},
        \end{split}
    \end{equation*}
    where $b'\sim b$ if $b-b'\in \{0,\pm 1\}^d$. Then, $B\subset [-1,1]^d$ and $|B|\le \eta$.

    If $a\in Q_{a'}$ and $b\in B'$, then
    \[
        a+\bigcup_{b'\sim b} Q_{b'} \supset Q_{a'+b'}.
    \]
    It follows that $A+\tilde{B}\supset [0,1]^d$ for all $A\in\mathcal{A}$. But $A+b\notin [0,1]^d$ unless $b\in [-1,1]^d$. We conclude that $A+B\supset [0,1]^d$ for all $A\in\mathcal{A}$ as well.
\end{proof}

\begin{corollary} \label{cor:discrete-to-continuous}
    Let $P,Q\subset\R^d$ be two cubes of the same side-length $r\in (0,1]$. Let $\e>0$ and let $\delta\in 2^{-\N}$ be a small dyadic number. Let $\mathcal{A}$ be a family of compact subsets of $P$. Suppose that
    \[
        |A|_{\delta}> \frac{(108)^d}{\e}\log\left(\delta^{-1}\cdot|\mathcal{A}|_{\delta}\right) \quad\text{for all }A\in\mathcal{A}.
    \]
    Moreover, let $a(A)\in A$ for all $A\in\mathcal{A}$.
    Then, there exists a set $B\subset 2Q$ which is a union of $\delta$-dyadic cubes, such that $|B|\le \e|Q|$ and
    \[
        A+B \supset a(A)+Q \quad\text{for all }A\in\mathcal{A}.
    \]
\end{corollary}
\begin{proof}
    By rescaling and translating, we may assume that $Q=[-1,1]^d$. Note that rescaling the rectangles has the effect of rescaling the value of $\delta$ by the same factor $2/r$ in both $|A|_{\delta}$ and $|\mathcal{A}|_{\delta}$, so the assumption of the corollary is preserved.

    For each $A\in\mathcal{A}$, we can cover $P$ by $2^d$ cubes of side length $1$ having $a(A)$ as a vertex. By pigeonholing, there exists one such cube $\tilde{P}_A$ satisfying
    \[
     |\widetilde{P}_A \cap A|_{\delta} \ge 2^{-d}|A|_{\delta}.
    \]
    Let $v(A)$ be the lower left corner of $\tilde{P}_A$. By Lemma \ref{lem:discrete-to-continuous} (applied to a suitable translation of the cubes) there exists a set $B'\subset [-2,0]^d$ which is a union of $\delta$-dyadic cubes, such that $|B'|\le 3^{-d}\e$ and
    \[
      A+B' \supset v(A)+[-1,0]^d \quad\text{for all }A\in\mathcal{A}.
    \]
    Let
    \[
        B = \bigcup_{v\in \{0,1,2\}^d} B'+\sum_{i=1}^d e_i v_i = \bigcup_{w,w'\in\{0,1\}^d} B'+\sum_{i=1}^d e_i w_i + \sum_{j=1}^d e_j w'_j.
    \]
    Since $a(A)=v(A)+\sum_{i=1}^d e_i w_i$ for some $w\in\{0,1\}^d$, we conclude that $A+B\supset a(A)+[0,1]^d$ for all $A\in\mathcal{A}$.
\end{proof}

\subsection{Main technical result on sums with nonempty interior}

In this section, we prove a technical result that gives a sufficient condition on a set $A$ and a family of functions $\mathcal{F}$ for the sumset $f(A)+B$ to have nonempty interior for all $f\in\mathcal{F}$ and some null set $B$.

Given a family of functions $\mathcal{F}$ on $[0,1]^d$, we let $|\mathcal{F}|_{\delta}$ denote the covering number of $\mathcal{F}$ in the $L^\infty([0,1]^d)$ norm, that is, the smallest number $M$ such that there are $f_1,\ldots,f_M\in\mathcal{F}$ such that $\min_{i=1}^M\|f-f_i\|_{L^\infty([0,1]^d)}\le \delta$ for all $f\in\mathcal{F}$. (If no such $M$ exists, we set $|\mathcal{F}|_{\delta}=\infty$ for completeness.)
\begin{theorem} \label{thm:technical-nonemptyinterior}
    Fix a constant $C\ge 1$. Let $N(\delta), M(\delta):(0,1]\to \N$ be functions such that
    \begin{equation} \label{eq:limit-N-M}
        \lim_{\delta\to 0^{+}} \,\frac{N(C\delta)}{\log \bigl(\delta^{-1}\cdot M(\delta)\bigr)} = +\infty.
    \end{equation}
    Let $\mathcal{F}$ be a family of functions from $[0,1]^d\to \R^d$ such that:
    \begin{enumerate}[label=\textup{(\Roman*)}]
        \item \label{it:I} 
        \[
            C^{-1}\|x-y\|_{\infty}\le \|f(x)-f(y)\|_\infty \le \|x-y\|_{\infty} \quad\text{for all }x,y\in [0,1]^d \text{ and all }f\in \mathcal{F}.
        \]
        \item  \label{it:II}
        For all sufficiently small $\delta>0$,
        \[
            |\mathcal{F}|_{\delta} \le M(\delta).
        \]
    \end{enumerate}

    Then, for any $N$-large set $A\subset [0,1]^d$, there exists a compact set $B$ with $|B|=0$ such that
    \[
        [0,1]^d\subset f(A)+B \quad\text{for all } f\in\mathcal{F}.
    \]
\end{theorem}
\begin{proof}
    By passing to a compact subset if needed, we may assume that $A=A'$ in the definition of $N$-large, that is, for any dyadic cube $Q$ such that $A\cap Q^{\circ}\neq\emptyset$, there exist arbitrarily small $\delta>0$ such that $|A\cap Q^{\circ}|_{\delta}\ge N(\delta)$.

    Fix $a_0\in A$. By the property \ref{it:II}, we can cover $\mathcal{F}$ by finitely many balls of radius $\delta$ in the $L^\infty([0,1]^d)$ norm (for some $\delta>0$). On each such ball, the value of $f(a_0)$ varies by at most $\delta$. Thus, the set $\{ f(a_0): f\in\mathcal{F}\}$ is bounded. We can therefore find a cube $R$ such that $[0,1]^d \subset f(a_0)+R$ for all $f\in\mathcal{F}$.

    By Lemma \ref{lem:nonemptyinterior}, it is enough to show that $(A,\mathcal{F})$ has the $\RRP$ property with initial rectangle $R$.

    Let, then $Q\subset R$ be a rectangle, let $\e>0$, and fix $a\in A$. By covering $Q$ by finitely many dyadic cubes, we may assume that $Q$ is itself a dyadic cube (indeed, if there are $L$ such cubes, we can apply this case with constant $\e/L$ and then take the union of the resulting elementary sets). We may also assume that $Q$ has side length at most $1$.

    Let $P=P(Q)$ be the cube centred at the origin and of side $C^{-1}$ times that of $Q$. Then, $A\cap (a+P^{\circ})\neq\varnothing$, and so there exist arbitrarily small $\delta>0$ such that
    \[
        |A\cap (a+P^{\circ})|_{\delta}\ge N(\delta).
    \]
    Thus, we can take a finite subset $A'=A'(a,Q)$ of $A\cap (a+P^{\circ})$ such that $a\in A'$ and $|A'|_{\delta}\ge N(\delta)$.

    Write $\delta'=\delta/C$. By \ref{it:II}, there exist $f_1,\ldots,f_{M(\delta')}\in\mathcal{F}$, such that
    \[
        \min\left\{\|f-f_i\|_{L^\infty([0,1]^d)}: 1\le i\le M(\delta')\right\} \le C^{-1}\cdot \delta.
    \]
    Consider the family $\mathcal{A}=\{f_i(A'):i=1,\ldots,M(\delta')\}$. By the bi-Lipschitz property \ref{it:I}, we have that for all $f\in\mathcal{F}$, the set $f(A')$ is $\delta'$-separated and contained in $f(a)+Q$.

    For each point $y_1=f_1(a)\in f_1(A)$, we have $|f_2(a)-f_1(a)|\le \|f_2-f_1\|_{L^\infty(A)}$. Thus, $f(A_2)$ is contained in the closed $\|f_2-f_1\|_{L^\infty(A)}$-neighbourhood of $f_1(A)$. By symmetry, we get
    \[
        \mathsf{H}(f_1(A),f_2(A))\le \|f_1-f_2\|_{L^\infty(A)},
    \]
    where we recall that $\mathsf{H}$ denotes the Hausdorff distance. Therefore, $|\mathcal{A}|_{\delta'}\le M(\delta')$. By the assumption \eqref{eq:limit-N-M} applied with $\delta'$ in place of $\delta=C\delta'$, if $\delta$ is chosen sufficiently small, we have
    \[
        |A'|_{\delta'}\ge |A'|_{\delta} \ge  N(\delta)> \frac{(108)^d}{\e}\log \bigl((\delta')^{-1}\cdot |\mathcal{A}|_{\delta'}\bigr).
    \]
    Corollary \ref{cor:discrete-to-continuous} then implies that there exists an elementary set $T\subset 2Q$ with $|T|\le \e|Q|$ and such that $f(A')+T\supset f(a)+Q$ for all $f\in\mathcal{F}$. We have verified the $\RRP$ property, and this completes the proof.
\end{proof}

\subsection{Applications}

\subsubsection{Affine images}

Let $\mathcal{F}$ be a compact subset of $\text{GL}_d(\R)$. Then there exists a constant $c=c(\mathcal{F})>0$ such that all $cf$, $f\in\mathcal{F}$, are contractions in the $\ell^\infty$ metric. By compactness, there is also a constant $C=C(\mathcal{F})\ge 0$ such that
\[
    C^{-1} \|x-y\|_{\infty} \le c\|f(x)-f(y)\|_{\infty} .
\]
Thus, after rescaling by $c$, we see that $\mathcal{F}$ satisfies the assumption \ref{it:I} of Theorem \ref{thm:technical-nonemptyinterior}. With this observation, we conclude the proof of Theorem \ref{thm:full-sumset-Erdos}.

\begin{proof}[Proof of Theorem \ref{thm:full-sumset-Erdos}]
    If $A$ is as in the statement then, by Lemma \ref{lem:large-Hausdorff-packing}, there exists $\eta>0$ such that $A$ is $N$-large for the function
    \[
        N(\delta) = \log^{1+\eta}(1/\delta).
    \]

    We can write $\text{GL}_d(\R)$ as a countable union of compact sets $\mathcal{F}_i$. Each of these sets can be covered by $\lesssim_i \delta^{-d^2}$ balls of radius $\delta$ in the $\ell^\infty$ metric. Thus, we can take $M(\delta)=C_i\delta^{-d^2}$ in Theorem \ref{thm:technical-nonemptyinterior}. The condition \eqref{eq:limit-N-M} is satisfied, and thus there is a set $B_i$ of zero Lebesgue measure such that $f(A)+B_i$ has nonempty interior for all $f\in\mathcal{F}_i$. (To be more precise, we scale the functions in each $\mathcal{F}_i$ by a constant $c_i$ so that assumption \ref{it:I} is satisfied, and also rescale the resulting sets $B_i$ provided by the theorem).

    Take
    \[
        B = \bigcup_i \bigcup_{q\in\mathbb{Q}} B_i + q.
    \]
    Then $B$ is an $F_\sigma$ set of zero Lebesgue measure, and $f(A)+B=\R^d$ for all $f\in\text{GL}_d(\R)$ and therefore for all non-singular affine maps.
\end{proof}

\subsubsection{Polynomial images}

\begin{corollary} \label{cor:polynomial-images}
    Let $A\subset [0,1]$ be a Borel set with $\plogdim(A)>2$ or $\hlogdim(A)>1$. Then there exists an $F_\sigma$ set $B\subset \R$ with $|B|=0$ such that $f(A)+B=\R$ for all non-constant polynomial maps $f:\R\to \R$.
\end{corollary}
\begin{proof}
    By Lemma \ref{lem:large-Hausdorff-packing}, there is $\eta>0$ such that $A$ is $N$-large for the function
    \[
        N(\delta) = \log^{(1+\eta)}(1/\delta).
    \]
    Replacing $A$ by the subset $A'$ in the definition of $N$-large, we may assume that $A\cap I$ is $N$-large for all dyadic intervals $I$ such that $A\cap I^\circ\neq\emptyset$. Let $\mathcal{I}(A)$ denote the set of all such intervals (of all lengths).

    For each $k\in\N$ and $I\in\mathcal{I}(A)$, let $\mathcal{P}_{k,I}$ be the family of all polynomials $f$ of degree in $[1,k]$ such that $f(x_I)=0$ (where $x_I$ is the left endpoint of $I$) and $|f'|_I| \ge k^{-1}$.

    It is easy to check that $\mathcal{P}_{k,I}$ can be covered by $\lesssim_{k,I} O(\delta^{-(k+1)})$ balls of radius $\delta$ in the $\ell^\infty$ metric. Indeed, given equispaced points $x_I=x_0,x_1,\ldots,x_{k}\in I$, the map $f\mapsto (f(x_0),\ldots,f(x_{k}))$ is injective, and smooth. Thus, we can take $M(\delta)=O_{k,I}(\delta^{-(k+1)})$ in Theorem \ref{thm:technical-nonemptyinterior}. The condition \eqref{eq:limit-N-M} is satisfied, and thus there is a set $B_{k,I}$ of zero Lebesgue measure such that $f(A\cap I)+B_{k,I}$ has nonempty interior for all $f\in\mathcal{P}_{k,I}$.

    On the other hand, for any non-constant polynomial $f$ there are $k,I$ such that $f(x)=f(x_I)+g(x)$, where $g\in\mathcal{P}_{k,I}$. Let
    \[
        B = \bigcup_{k\in\N}\bigcup_{I\in\mathcal{I}(A)} \bigcup_{q\in\mathbb{Q}} B_{k,I} + q.
    \]
    Then $B$ is $F_\sigma$, and $f(A)+B=\R$ for all non-constant polynomials $f$.
\end{proof}

\subsubsection{Convex images}

\begin{corollary}
    Let $A\subset [0,1]$ be a Borel set with $\pdim(A)>1/2$. Then there exists an $F_\sigma$ set $B\subset \R$ with $|B|=0$ such that $f(A)+B=\R$ for all strictly increasing convex functions $f:[0,1]\to \R$.
\end{corollary}
\begin{proof}
    By the assumption and Lemma \ref{lem:large-Hausdorff-packing}, there is $\eta>0$ such that $A$ is $N$-large for the function
    \[
        N(\delta) = \delta^{-(1/2+\eta)}.
    \]

    It was proved in \cite[Proposition 6.1]{SS15} that the family $\mathcal{C}^+$ of increasing convex functions $f:[0,1]\to [0,1]$ with $f(0)=0$ and right-derivative bounded above by $1$ satisfies
    \[
        |\mathcal{C}^+|_{\delta} \le M(\delta) := \exp\left(O(\delta^{-1/2})\log(\delta)\right).
    \]
    Note that $M$ and $N$ satisfy the assumptions of Theorem \ref{thm:technical-nonemptyinterior}. By a simple rescaling argument, if we bound the right-derivative of $f$ by $K$ instead, we get the same bound up to a constant. On the other hand, if $f$ is convex and strictly increasing, then there are $K$ and a dyadic interval $I$ such that $A\cap I$ is $N$-large and $f|_I$ is lower-Lipschitz with constant $K^{-1}$. A similar argument to that of Corollary \ref{cor:polynomial-images} then gives the claim.
\end{proof}

\section{Positive measure}

\label{sec:full-measure}

\subsection{The case of positive dimension}

As mentioned in the introduction, the special case of Theorem \ref{thm:full-measure-sumset} in which $A$ has positive Hausdorff dimension is very likely known. We include a proof for completeness.
\begin{lemma} \label{lem:positive-measure-positive-dim}
    Given $s>0$, there exists an $F_\sigma$ set $B\subset \R^d$ with $\hdim(B)<d$ such that $|\R^d\setminus(A+B)|=0$ for all Borel sets $A\subset \R^d$ with $\hdim(A)\ge s$.
\end{lemma}
\begin{proof}
    J.P.~Kahane \cite{Kahane66} showed that for any $\kappa<d/2$ there exists a Borel probability measure $\mu$ on $\R^d$ such that $\hdim(\supp\mu)<d$ and $|\widehat{\mu}(\xi)| \le C\cdot|\xi|^{-\kappa}$ for some constant $C$ and all frequencies $\xi\in\R^d$. Kahane's constructions are random, see \cite{FraserHambrook23} for a recent deterministic construction.

    Given $s>0$, take $\kappa=1/2-s/4$. Suppose $A$ is a Borel set with $\hdim(A)\ge s$. By Frostman's lemma \cite[Theorem 2.8]{Mattila15} and the Fourier interpretation of the Riesz energy \cite[Theorem 3.10]{Mattila15}, there exists a Borel probability measure $\nu$ supported on $A$ such that
    \[
        \int |\widehat{\nu}(\xi)|^2 |\xi|^{s+s/8-d}\,d\xi <\infty.
    \]
    By Plancherel and the convolution formula, this implies that $\mu*\nu$ is absolutely continuous with a density in $L^2$, and in particular $|A+\supp(\mu)|>0$. We can therefore take
    \[
        B=\bigcup_{q\in\mathbb{Q}^d} \supp(\mu)+q.
    \]
\end{proof}
There are several obstacles in extending this approach to sets of zero Hausdorff dimension. We do not know how to get a suitable Fourier expression for the Riesz energy corresponding to a function $\psi$ that grows more slowly than any power. The known constructions of sets with fast Fourier decay of dimension close to $1$ entail logarithmic losses in the decay rate (compared to the theoretical fastest decay). Even a logarithmic loss would leave out many sets of dimension zero. Hence, we follow a somewhat different approach, by constructing sets which locally have the fastest possible Fourier decay. This fast local decay is derived from number-theoretic constructions, which we review in the next section.

\subsection{Preliminaries on discrete sumsets and the discrete Fourier transform}

For the reader's convenience, we recall some basic facts about the Fourier transform on finite Abelian groups $G$. See \cite[Chapter 4]{TaoVu06} for more details.

The group $G$ can be written as $\Z_{m_1}\times\cdots\times \Z_{m_r}$ for some cyclic groups $Z_{m_j}$. Given $x=(x_1,\ldots,x_r)$ and $\xi=(\xi_1,\ldots,\xi_r)\in G$, we define a bilinear form from $G\times G\to \R/\Z$ by
\[
    x\cdot \xi = \sum_{j=1}^r x_j\xi_j/m_j.
\]
We then define the Fourier transform $\hat{f}:G\to \mathbb{C}$ by
\[
    \mathcal{F}(f)(\xi) = \hat{f}(\xi) = \frac{1}{|G|}\sum_{x\in G} f(x) e^{-2\pi i \xi\cdot x}.
\]
(Here, $|G|^{-1}$ is a choice of normalization that is not essential; we follow the convention of \cite{TaoVu06}.) One can define the Fourier transform more abstractly in terms of the dual group $\hat{G}$, but we will not need this here. Different choices of bilinear forms give different Fourier transforms, but they are all equivalent up to a group isomorphism.

The $L^2$ norm of $f:G\to\C$ is given by
\[
    \|f\|_2 = \left(\sum_{x\in G} |g(x)|^2\right)^{1/2}.
\]
We have the Plancherel identity
\[
    |G| \cdot \|\hat{f}\|_2^2 = \|f\|_2^2.
\]
The \emph{convolution} of two functions $f,g:G\to \mathbb{C}$ is defined by
\[
    (f*g)(x) = \frac{1}{|G|}\sum_{y\in G} f(y)g(x-y).
\]
We have the convolution formula:
\[
    \widehat{f*g}(\xi) = \hat{f}(\xi)\hat{g}(\xi).
\]

If $A\subset G$, the indicator function $\mathbf{1}_A$ takes value $1$ on $A$ and $0$ off it. Note that $\widehat{\mathbf{1}_A}(0) = |A|/|G|$, and
\begin{equation} \label{eq:support-conv}
    \mathbf{1}_A*\mathbf{1}_B(x)=0 \text{ for } x\in G\setminus (A+B).
\end{equation}
The \emph{linear bias} (or Fourier uniformity) of $B\subset G$ is defined as
\[
    \|B\|_u = \max_{\xi\neq 0} |\widehat{\mathbf{1}_B}(\xi)|.
\]
The following lemma is standard (see e.g. \cite[Exercise 4.3.7]{TaoVu06}), but we include a proof for completeness.
\begin{lemma} \label{lem:bias-to-sumset}
   For any $A,B\subset G$, we have
   \[
     \frac{|G|}{|A+B|} \le 1+\frac{\|B\|_u^2 |G|^3}{|A||B|^2}.
   \]
\end{lemma}
\begin{proof}
By \eqref{eq:support-conv}, the Cauchy-Schwarz inequality (or just the convexity of $t\mapsto t^2$), and Fubini,
\begin{equation} \label{eq:application-CS}
    \frac{|A|^2|B|^2}{|G|^2} = \left(\sum_{x\in A+B} \mathbf{1}_A*\mathbf{1}_B(x)\right)^2 \le |A+B| \|\mathbf{1}_A*\mathbf{1}_B\|_2^2.
\end{equation}

On the other hand, by Plancherel and the convolution formula,
\begin{align*}
    |G|^{-1}\cdot \|\mathbf{1}_A*\mathbf{1}_B\|_2^2 &= \left\|\mathcal{F}\left(\mathbf{1}_A*\mathbf{1}_B\right)\right\|_2^2  =  \left\| \widehat{\mathbf{1}_A}\cdot \widehat{\mathbf{1}_B} \right\|_2^2 \\
 &\le \widehat{\mathbf{1}}_A(0)^2 \widehat{\mathbf{1}_B}(0)^2 + \|B\|_u^2 \sum_{\xi\neq 0} |\widehat{\mathbf{1}}_A(\xi)|^2 \\
 &\le \frac{|A|^2|B|^2}{|G|^4}  + \|B\|_u^2 \|\widehat{\mathbf{1}_A}\|_2^2 \\
 &= \frac{|A|^2|B|^2}{|G|^4} + \frac{\|B\|_u^2\|\mathbf{1}_A\|_2^2}{|G|} \\
 &= \frac{|A|^2|B|^2}{|G|^4} + \frac{\|B\|_u^2 |A|}{|G|}.
\end{align*}
Recalling \eqref{eq:application-CS} and rearranging, we obtain the claim.
\end{proof}

The following result on the existence of sets with very small linear bias is classical. See \cite[Theorem 6.8]{BabaiFourier} (note that the Fourier transform there is not normalized by dividing by $|G|$ as in our case), or \cite[Lemma 4.14]{TaoVu06} for a special case. For a prime power $q$, we let $\mathbb{F}_q$ denote the finite field with $q$ elements.
\begin{lemma} \label{lem:Gauss-sum}
    Let $q$ be a prime power. Suppose $k\mid q-1$ and let $B = \{ x^k: x\in \mathbb{F}_q\}$. Then $|B|=(q-1)/k$, and $\|B\|_u < 1/\sqrt{q}$.
\end{lemma}
Here, the Fourier transform is taken with respect to the additive group structure of $\mathbb{F}_q$.
Using this lemma, we deduce the following crucial result, which can be seen as a discrete version of Theorem \ref{thm:full-measure-sumset}. Although it is based on well-known facts, we have not been able to find it in the literature.
\begin{proposition} \label{prop:large-sumset}
    Fix $\eta\in (0,1/3]$. For every $m_0\in\N$, there exist an integer $m_0\le m \le 4^{1/\eta} m_0$ and a set $B\subset\Z_m^d$, such that $|B|\le \eta\cdot m^d$, and
    \[
      \frac{m^d}{|A+B|} \le 1+\frac{1}{4\eta^2 |A|} \quad\text{for all } A\subset \Z_m^d.
    \]
\end{proposition}
\begin{proof}
    By Bertrand's postulate, there is a prime $k\in [\eta^{-1},2\eta^{-1}]$. By Fermat's Little Theorem,
    \begin{equation} \label{eq:FLT}
        2^{d s(k-1)}\equiv 1 \bmod k,
    \end{equation}
    where $s$ satisfies
    \[
        m_0 \le 2^{s(k-1)} < 2^k m_0 < 4^{1/\eta} m_0.
    \]
    Write $m=2^{s(k-1)}$ and $q= 2^{d s(k-1)}=m^d$ for simplicity. Since $k\mid q-1$ by \eqref{eq:FLT}, we can apply Lemma \ref{lem:Gauss-sum}. Namely, if $B=\{ x^k: x\in \mathbb{F}_q\}$, then
    \[
        \|B\|_u \le \frac{1}{\sqrt{q}}.
    \]
    To be precise, here the uniformity is with respect to the Fourier transform on $\mathbb{F}_{q}$, but as an additive group it is isomorphic to $\Z_m^d$, and so the uniformity carries over.

    Since $|B|\ge \eta q/2$, we deduce that
    \[
        \frac{\|B\|_u^2 \cdot q^3}{|B|^2} \le \frac{1}{4\eta^{2}}.
    \]
    In light of Lemma \ref{lem:bias-to-sumset}, this yields the claim.
\end{proof}

\subsection{Continuous analogues, and the proof of Theorem \ref{thm:full-measure-sumset}}

\begin{corollary} \label{cor:large-sum-Z^d}
    For every $\eta,\e\in (0,1/3]$, there is $N=N(\eta,\e)\in\N$ such that the following holds for infinitely many $m\in\N$.

    There exists a set $B\subset \{ -m,\ldots, m-1\}^d$ with $|B|\le \eta\cdot  m^d$, such that
    \[
        \left|(A+B)\cap [m]^d\right| \ge (1-\e) m^d \quad\text{for all } A\subset [m]^d \text{ with } |A|\ge N.
    \]
\end{corollary}
\begin{proof}
    By Proposition \ref{prop:large-sumset}, there are $N=N(\eta,\e)$, infinitely many $m$, and sets $B'\subset \Z_m^d$ such that $|B'|\le \eta 2^{-d} m^d$ and
    \[
        |A+B'| \ge (1-\e) m^d \quad\text{for all } A\subset \Z_m^d \text{ with } |A|\ge N.
    \]
    Identify $B'$ with a subset of $[m]^d$, and let 
    \[
        B= \bigcup\left\{ B'- m \sum_{i=1}^d  u_i e_i : u\in \{0,1\}^d \right\}.
    \]
    Take $A\subset [m]^d$ with $|A|\ge N$. Then, seeing $A$ as a subset of $\Z_m^d$, we have $|A+B'|\ge (1-\e)m^d$, where $A+B'$ is computed in $\Z_m^d$. Note that for any $a\in A$ and $b\in B'$, the sum $a+b$ in $\Z^d$ lies in $[2m]^d$, and thus there is $u\in\{0,1\}^d$ such that $c=  a+b - m\sum_{i=1}^d u_i e_i \in [m]^d$ satisfies $c\equiv a+b \bmod \Z_m^d$. Since $B\subset \{ -m,\ldots, m-1\}^d$ by definition, the claim follows.
\end{proof}

We deduce a continuous version of Corollary \ref{cor:large-sum-Z^d} by an argument very similar to those in Lemma \ref{lem:discrete-to-continuous} and Corollary \ref{cor:discrete-to-continuous}.
\begin{corollary} \label{cor:large-sum-R^d}
    For every $\eta,\e\in (0,1/3]$, there is $N=N(\eta,\e)\in\N$ such the following holds for all small enough $\delta>0$.

    Fix a dyadic scale $2^{-m'}$ and a cube $Q\in\cD_{m'}$. There exists a set $B\subset 4\cdot Q$ with $|B|\le \eta\cdot |Q|$, such that
    \[
        \bigl|(a_0+Q)\cap (A+B)\bigr| \ge (1-\e)|Q|
    \]
    for all  $A\subset [0,2^{-m'}]^d$ with  $|A|_{\delta}\ge N$ and all $a_0\in A$.

    Moreover, $B$ can be taken to be a union of closed dyadic cubes of the same size, depending only on $d,\eta,\e,\delta,m'$.
\end{corollary}
\begin{proof}
    By rescaling and translating, we may assume that  $A\subset [0,2]^d$, and $Q=[-1,1]^d$.

    Suppose first that $\delta=1/m$, where $m$ is one of the values to which Corollary \ref{cor:large-sum-Z^d} applies with $6^{-d}\eta$ in place of $\eta$ and $\widetilde{N}:=C_d 2^d N$ in place of $N$, where $C_d \geq 1$ is a constant to be chosen later. (This defines $\widetilde{N}$, and therefore also $N=C_d^{-1} 2^{-d}\widetilde{N}$.)

    By Corollary \ref{cor:large-sum-Z^d}, there exists $B'\subset \{-m,\ldots, m-1\}^d$ with $|B'|\le \eta 6^{-d}\cdot m^d$, such that
    \begin{equation} \label{eq:large-sum-Z^d}
        |A'+B'| \ge (1-\e) m^d \quad\text{for all } A\subset [m]^d \text{ with } |A'|\ge \widetilde{N}.
    \end{equation}
    Let $Q_v$ be given by \eqref{eq:Qv} and define
    \begin{align*}
        B_1 &= \bigcup_{b\in B'} \bigcup_{b'\sim b} Q_{b'},\\
        B_2 &= \bigcup_{u\in\{0,1\}^d} B_1+\sum_{i=1}^d u_i e_i,\\
        B &= \bigcup_{v\in\{0,1\}^d} B_2 +\sum_{i=1}^d v_i e_i,
    \end{align*}
    where $b'\sim b$ if $b-b'\in \{0,\pm 1\}^d$. Then, $B\subset [-2,4]^d\subset 4\cdot Q$, and $|B|\le \eta$.

    Suppose $A\subset [0,2]^d$ has $|A|_{\delta}\ge N$, and fix $a_0\in A$. Covering $[0,2]^d$ by $2^d$ cubes of side length $1$ having $a_0$ as a vertex, we can pigeonhole one such cube, say $Q_0$, such that
    \[
        |A\cap Q_0|_{\delta} \ge 2^{-d} N .
    \]
    Let $v_0$ be the lower-left corner of $Q_0$.
    Let $A'$ be obtained from $(A\cap Q_0)-v_0\subset [0,1]^d$ by replacing each element in $\bigl[(A\cap Q_0)-v_0\bigr]\cap Q_{a'}$ by $a'$. Then $A'\subset [m]^d$. Moreover, since each $Q_{a'}$ can be covered by $C_d$ balls of radius $\delta$ for some constant $C_d$ depending only on $d$, we have $|A'| \ge C_d^{-1} |A\cap Q_0|_{\delta} \ge \widetilde{N}$.

    If $a\in Q_{a'}$ and $b\in B'$, then
    \[
        a+\bigcup_{b'\sim b} Q_{b'} \supset Q_{a'+b'}.
    \]
    It follows from \eqref{eq:large-sum-Z^d} that
    \[
        \bigl|(v_0+[0,1]^d)\cap (A+B_1)\bigr| \ge (1-\e) m^{-d}|A'+B'|\ge 1-\e.
    \]
    Since $v_0, a_0$ are both vertices of $Q_0$ and $v_0$ is the lower-left corner, we have $a_0=v_0+\sum_{i=1}^d u_i e_i$ for some $u\in\{0,1\}^d$, we deduce that
    \[
        \bigl|(a_0+[0,1]^d)\cap (A+B_2) \bigr| \ge 1-\e.
    \]
    Finally, it follows from the definition of $B$ that
    \[
        \bigl|(a_0+Q)\cap (A+B)\bigr| \ge (1-\e)2^d = (1-\e)|Q|.
    \]
    If $\delta$ is smaller than $1/m_0$, where $m_0$ is the smallest value of $m$ such that the above argument applies, then $\delta
    \ge 1/m$ some $m>m_0$ for which Corollary \ref{cor:large-sum-Z^d} holds, and the claim follows since $|A|_{1/m}\ge |A|_{\delta}$.
\end{proof}

\begin{proposition}
    Let $\phi\in\Phi$ and $\eta,\e>0$. There exists a compact set $B$ (depending only on $\phi$, $\eta$, $\e$ and the ambient dimension $d$) with $|B|=0$, such that
    \[
        \bigl|[0,1]^d\cap (A+B)\bigr| \ge 1-\e,
    \]
    for all Borel sets $A\subset [0,1]^d$ with $\mathcal{H}^{\phi}(A)>\eta$.
\end{proposition}
\begin{proof}
    The proof is similar to that of Lemma \ref{lem:nonemptyinterior}. Invoking Lemma \ref{lem:Hausdorff-unif-large}, we may assume that $A$ is $(N_k,\delta_k)$-uniformly large for some $N_k\to\infty$ and $\delta_k$ that depend only on $\phi$ and $\eta$.

    Let $\mathcal{A}$ denote the family of all sets $A$ with the following property: for any $k$ and any $Q\in\mathcal{D}_k(A)$, we have
    \begin{equation} \label{eq:good-assumption}
        |A\cap Q|_{\delta_k} \ge N_k.
    \end{equation}
    By the definition of uniformly large (Definition \ref{def:large-set}), it is enough to prove the claim for $A\in\mathcal{A}$.

    We will inductively construct a sequence of sets $B_j\subset [-1,1]^d$ and a strictly increasing sequence of integers $m_j\in\N$ such that:
    \begin{enumerate}[label=(\alph*)]
        \item \label{it:a-PM} $B_j$ is a union of closed cubes $\{Q^{(j)}_s\}_{s\in S_j}$, where $Q^{(j)}_s\in\cD_{m_j}$.
        \item \label{it:b-PM} $B_{j+1}\subset B'_j$, where
        \[
            B'_j = \bigcup_{s\in S_j} 4 Q ^{(j)}_s.
        \]
        \item \label{it:c-PM} $|B_j|\le  2^{-j}$,
        \item \label{it:d-PM} For all $A\in\mathcal{A}$, there is a finite subset $A'_j\subset A$ such that $\left| [0,1]^d\setminus (A'_j+B_j)\right|\le (1-2^{-j})\e$.

    \end{enumerate}
    Assuming the construction can be carried out, let
    \[
        B''_j = \bigcup_{s\in S_j} 6 Q^{(j)}_s,
    \]
    and note from \ref{it:a-PM}--\ref{it:b-PM} that $B''_j$ is a nested sequence. Let $B=\bigcap_{j=1}^\infty B''_j$. Then $|B|=0$ by \ref{it:c-PM}, since $|B''_j|\le 6^d |B_j|$. If $A\in\mathcal{A}$, we have
    \[
        \bigl|[0,1]^d\setminus (A+B)\bigr| = \lim_{j\to\infty} \left|[0,1]^d\setminus (A+B''_j)\right| \le \e.
    \]
    Thus, the existence of the claimed set $B$ follows from the construction.

    Set $B_0=[-1,1]^d$, $m_0=0$ and $A'_0$ given by an arbitrary element of $A$ for all $A\in\mathcal{A}$. Suppose $B_j$, $m_j$ have been constructed. Suppressing the dependence on $j$ for simplicity, let
    \[
      B_j = \bigcup_{s\in S} Q_s, \quad Q_s\in\mathcal{D}_{m}.
    \]
    Let
    \[
        N' = N(1/2,2^{-(j+1)}\e),
    \]
    where $N$ is the function from Corollary \ref{cor:large-sum-R^d}.

    Since $N_k\to\infty$ and $\delta_k\to 0$, we can pick $k\ge m$ so that $N_k\ge N'$ and $\delta_k$ is small enough that the conclusion of Corollary \ref{cor:large-sum-R^d} holds.

    Let
    \[
        B_{j+1} = \bigcup_{s\in S} B'_s,
    \]
    where $B'_s$ is the outcome of applying Corollary \ref{cor:large-sum-R^d} to $Q_s$.  Claims \ref{it:a-PM}, \ref{it:b-PM} and \ref{it:c-PM} hold by Corollary \ref{cor:large-sum-R^d} and the choice of $N'$.

    Fix $a_0\in A'$. Then, since $A\in\mathcal{A}$ and $k\ge m$, we know that $|A\cap R_0|_{\delta}\ge N'$, where $R_0$ is the dyadic cube in $\cD_{m}$ containing $a_0$. Let $A''(a_0)$ be a finite subset of $A\cap R_0$ such that $|A''(a_0)|_{\delta}\ge N'$. Then, by Corollary \ref{cor:large-sum-R^d},
    \[
        \bigl| (a_0+Q_s)\cap (A''(a_0) + B'_s) \bigr| \ge \left(1-2^{-j-1}\e\right)|Q_s|.
    \]
    Adding up over all $s$, we get that
    \[
        \bigl| (A'_j  + B_j) \setminus (A'_{j+1} + B_{j+1})\bigr| \le 2^{-(j+1)}\e.
    \]
    We conclude from the inductive assumption \ref{it:d-PM} that
    \[
        \bigl| [0,1]^d\setminus (A'_{j+1}+B_{j+1})\bigr| \le (1-2^{-j})\e + 2^{-(j+1)}\e = (1-2^{-(j+1)})\e.
    \]
    This shows that \ref{it:d-PM} continues to hold for $j+1$, and the induction is complete.
\end{proof}

We can now complete the proofs of Theorem \ref{thm:full-measure-sumset} and Corollary \ref{cor:full-measure-sumset}.

\begin{proof}[Proof of Theorem \ref{thm:full-measure-sumset}]
    For any $n$, Proposition \ref{prop:large-sumset} gives a set $B_n$ with $|B_n|=0$ such that
    \[
        \bigl|[0,1]^d\cap (A+B_n)\bigr| \ge 1-1/n
    \]
    for all Borel sets $A\subset [0,1]^d$ with $\mathcal{H}^{\phi}(A)>1/n$. Let
    \[
    B=\bigcup_n \bigcup_{v\in\Z^n} B_n+v.
    \]
    Then $B$ is a null $F_\sigma$ set in $\R^d$ such that satisfies the claim of the theorem.
\end{proof}

%\bibliographystyle{plain}
%\bibliography{biblio}

\end{document}